\documentclass[a4paper,10pt]{amsart}



\usepackage{amssymb}
\usepackage{amsmath}
\usepackage{amsthm}
\usepackage{color}
\usepackage{enumerate}
\usepackage{enumitem}

\newcommand{\mat}[4]{ \left ( \begin{array}{cc} #1 & #2 \\ #3 &
      #4 \end{array} \right)} 

\newcommand{\pair}[1]{\langle #1 \rangle}

\newcommand{\End}{\operatorname{End}}

\newcommand{\Ext}{\operatorname{Ext}}

\newcommand{\Ker}{\operatorname{Ker}}

\newcommand{\Img}{\operatorname{Im}}

\DeclareMathOperator{\supp}{supp}

\newcommand{\Modr}[1]{\mathrm{Mod}\textrm{-}{#1}}

\newcommand{\Modl}[1]{{#1}\textrm{-}\mathrm{Mod}}

\DeclareMathOperator{\RCP}{RCP}

\DeclareMathOperator{\Z}{Z}

\theoremstyle{plain}
\newtheorem{theorem}{Theorem}[section]
\newtheorem{lemma}[theorem]{Lemma}
\newtheorem{proposition}[theorem]{Proposition}
\newtheorem{corollary}[theorem]{Corollary}
\newtheorem{question}[theorem]{Question}

\newtheorem{definition}[theorem]{Definition}

\newtheorem{remark}[theorem]{Remark}

\newtheorem{example}[theorem]{Example}
\newtheorem{examples}[theorem]{Examples}

\usepackage{tikz}
\usepackage{tikz-cd}
\usepackage{hyperref}

\title{STRONGLY EXCHANGE RINGS}

\author{Manuel Cort\'es-Izurdiaga}
\address{Departamento de Matemática Aplicada, Universidad de Málaga, 29071, Málaga, Spain}
\email{mizurdiaga@uma.es} \thanks{The first author is partially supported by the Spanish Government under grants PID2020-113552GB-I00 which include FEDER funds of the EU., and by Junta de Andaluc\'{\i}a under grant P20-00770}
 \author[Guil Asensio]{Pedro A. Guil Asensio}
 \address{Departamento de Matem\'aticas, Universidad
   de Murcia, Murcia, 30100, Spain} \email{paguil@um.es} \thanks{The
   second author is partially supported by the Spanish Government under
   grant PID2020-113206GB-I00/AEI/10.13039/501100011033 which includes FEDER funds of the EU, and by Fundaci\'on S\'eneca of
   Murcia under grant
  19880/GERM/15}

\begin{document}

\begin{abstract}
Two elements $a,b$ in a ring $R$ form a right coprime pair, written $\pair{a,b}$, if $aR+bR=R$. Right coprime pairs have shown to be quite useful in the study of left cotorsion or exchange rings. In this paper, we define the class of strongly right exchange rings in terms of  descending chains of them. We show that they are semiregular and that this class of rings contains left injective, left pure-injective, left cotorsion, local and left continuous rings. This allows us to give a unified study of all these classes of rings in terms of the behaviour of descending chains of right coprime pairs. 
\end{abstract}

\maketitle

\section{Introduction and notation}

The exchange property was introduced by Crawley and Jonsson in 1964 in their study of the decomposition properties of algebraic systems \cite[Definition 3.1]{CJ} and was later extended to arbitrary modules by Warfield in 1972 \cite{Warfield-Exchange}. He defined that a right module $M$ satisfies the (finite) exchange property if for any right module $X$ and decompositions  
$X=M'\oplus Y=\oplus_I N_i$ with $M'\cong M$ (resp., with $I$ finite), there exist submodules $N'_i\leq N_i$ such that $X=M'\oplus(\oplus_I N'_i)$. 
It is well known that the finite exchange property implies the general exchange property for finitely generated modules, but it is not known if this is true for every module. This is one of the oldest open questions in Ring Theory, whose origin can be traced back to the pioneering work of 
Crawley and Jonsson (see \cite[p.807]{CJ}). 
Warfield proved that a module $M$ satisfies the finite exchange property if and only if so does its endomorphism ring. He also proved that this property is left-right symmetric for rings and called them exchange rings \cite{Warfield-Exchange}. Exchange rings were later independently characterized by Goodearl \cite[p. 167]{Goodearl-Warfield} and Nicholson \cite{Nicholson} as those rings $R$ satisfying that for any element $r\in R$, there exists an idempotent $e\in R$ such that $e\in rR$ and $1-e\in (1-r)R$. See \cite{Lam-Equation,Lam-Crash} for a more detailed exposition on this topic and its relation with algebraic equations.

On the other hand, we recall that a pair $a,b$ of elements in a commutative PID, $R$, is called coprime if they do not have common divisors. This is equivalent to say that they satisfy the Bezout identity $ar+bs=1$, for some $r,s \in R$. This definition has been recently extended to arbitrary (not necessarily commutative) rings in a completely different context \cite{GuilHerzog}. Namely, a pair of elements $a,b$ in a ring $R$ is called right coprime, denoted $\pair{a,b}$, when they satisfy the former condition; and an order relation was defined among them: $\pair{a,b}\leq \pair{a',b'}$ if and only if $aR\subseteq a'R$ and $bR\subseteq b'R$. It was shown in \cite{GuilHerzog} that this order relation is a quite useful tool to construct idempotent elements in rings. Using these ideas, the authors have proved that any left cotorsion ring is semiregular (that is, $R$ is von Neumann regular modulo its Jacobson radical $J(R)$ and idempotents in $R/J(R)$ can be lifted to idempotents in $R$, see \cite{Nicholson-Semiregular}) and $R/J(R)$ is left self-injective. In particular, they have given an arithmetic (and probably more conceptual) proof of the fact that endomorphism rings of pure-injective modules enjoy the former properties (see e.g. \cite{ZimmermannZimmermann}).  
 
The main tool in most of the proofs in \cite{GuilHerzog} is that certain descending chains of right coprime pairs (the compatible ones, see Definition \ref{d:CompatibleSystem}) have minimal lower bounds (with respect to the aforementioned order relation), and these lower bounds are of the form $\pair{e,1-e}$ for some idempotent element $e$. Let us note that this fact connects Goodearl and Nicholson characterization of exchange rings with coprime pairs: a ring $R$ is an exchange ring if and only if for every element $r\in R$, 
there exists a minimal right coprime pair below the right coprime pair $\pair{r,1-r}$. This key observation is the inspiration idea of the present paper. We call rings satisfying the above descending property  \textit{right strongly exchange rings} and study their main properties.
We show that the class of right strongly exchange rings contains left cotorsion rings (hence, left self-injective and left pure-injective rings), left perfect rings, local rings and left continuous rings. Moreover, right strongly exchange rings are semiregular. Consequently, this right strongly exchange property allows to unify the treatment of all the above classes of rings, which had been shown to be semiregular as well.
 
Let us outline the contents of this paper. We begin by obtaining in Section 2 the basic properties of right coprime pairs. We also give in this section several examples that show the behaviour of right coprime pairs in different rings, and observe that this behaviour characterize local, left perfect, von Neumann regular and exchange rings (see Proposition~\ref{p:CharacterizationRingsCoprimePairs}). 


In Section 3, we define right strongly exchange rings in terms of the central notion of compatible descending chain of right coprime pairs. By the preceding comments, right strongly exchange rings satisfy the exchange property, although there exist exchange rings which are not right strongly exchange (see Example \ref{e:RegularRingNotDescendingChain}). We also show that the class of  right strongly exchange rings include local, left cotorsion and left continuous rings (Proposition \ref{p:LocalAreStrongly} and Theorems \ref{t:CotorsionAreStrongly} and \ref{t:ContinuousHaveCompatibleLowerBounds}).

We study in Section 4 the main properties of right strongly exchange rings. We prove in Theorem \ref{t:ChainsImpliesSemiregular} that they are semiregular. If we assume the stronger property that compatible descending systems of right coprime pairs have minimal lower bounds, then they are left continuous modulo their Jacobson radical (Theorem \ref{t:StronglyImpliesContinuousModuloTheRadical}). We also give an example showing that the strongly exchange property is not left-right symmetric (Example \ref{e:NonSymmetric}). We close the paper by showing that a ring is semiperfect provided that it has a countable number of idempotents and is strongly exchange (Theorem~\ref{t:semiperfect}).
	
Let us fix some notation. For any set $A$, $|A|$ will denote its cardinality. The first infinite ordinal is denoted by $\omega$. Moreover, $R$ will denote an associative ring with identity that we fix for the whole paper. Module will mean right $R$-module and, when dealing with left modules, we will specify the side where scalars act. We denote by $\Modr R$ and $\Modl R$ the categories of right and left $R$-modules, respectively. As usual, we use the notation $R_R$ (resp. ${_R}R$) when we consider $R$ as a right (resp. left) $R$-module  over itself. Morphisms will act on the opposite side of scalars. Consequently, if $M$, $N$ and $L$ are modules and $f:M \rightarrow N$ and $g:N \rightarrow L$ are homomorphisms of modules, we denote their composition by $gf$ in case $M,N,L \in \Modr R$ and by  $fg$ if $M,N,L \in \Modl R$. We will denote by $J(R)$ the Jacobson radical of $R$. Given a family of modules, $\{M_i\mid i \in I\}$, we treat elements in $\prod_{i \in I}M_i$ as maps $x:I \rightarrow \bigcup_{i \in I}M_i$ with $x(i) \in M_i$ for each $i \in I$. The support of an element $x$ in the product is
\begin{displaymath}
\supp(x) = \{i \in I\mid x(i) \neq 0\}.
\end{displaymath}

Recall that a submodule $K$ of a module $M$ is called superfluous, written $K << M$, if there does not exist a proper submodule $L$ of $M$ with $K+L=M$. Dually, $K$ is called essential in $M$ if there does not exist a non-zero submodule $L$ with $K \cap L = 0$. A family of submodules $\{M_i\mid i \in I\}$ of $M$ is called independent if for every $i \in I$, $M_i \cap \sum_{j \neq i}M_j =0$. Given $X \subseteq M$ and $A \subseteq R$, we denote by $\textbf{l}_M(A)$ and $\textbf{r}_R(X)$ the corresponding annihilators of $A$ in $M$ and of $X$ in $R$, respectively. Notice that, if $e$ is an idempotent element, $\mathbf l_R(e) = R(1-e)$ and $\mathbf r_R(e) = (1-e)R$.

%
%
%
%
%
%
%

\section{Coprime pairs. Basic properties}

We begin this section by recalling the following definition from \cite{GuilHerzog}.

\begin{definition}
A right coprime pair in $R$ is a pair of elements, $a,b$, such that $R=aR+bR$. We denote the coprime pair by $\langle a,b\rangle$.
\end{definition}

We are interested in the cyclic right ideals generated by the elements in the right coprime pair. This is the reason why we define the following equivalence relation between them:
\begin{displaymath}
\pair{a,b} \sim \pair{a',b'} \Leftrightarrow aR = a'R \textrm{ and } bR = b'R
\end{displaymath}
We denote by $\RCP(R)$ the set of all equivalence classes of right coprime pairs under this equivalence relation. From now on, the term "right coprime pair" and the notation $\pair{a,b}$ will refer to one of these equivalence classes. Given a right coprime pair $\pair{a,b}$, any pair of elements $(a',b')$ of $R^2$ such that $\pair{a',b'}=\pair{a,b}$ will be called a \textit{pair of generators} of the coprime pair.

We may define in $\RCP(R)$ an order relation. Given two right coprime pairs $\langle a,b \rangle$ and $\pair{a',b'}$, we will say that
\begin{displaymath}
\pair{a,b} \leq \pair{a',b'} \textrm{ if and only if } aR \leq a'R \textrm{ and } bR \leq b'R.
\end{displaymath}

As we will see later on, this order relation is quite useful to find idempotent elements in $R$.

\begin{examples}\label{e:ExamplesCoprimesPairs}
\begin{enumerate}
\item For any $a \in R$, we have the coprime pair $\langle a,1-a\rangle$. We call this coprime pair \textit{basic}.

\item For any $a \in R$ and any unit $u \in R$, we have the coprime pair $\langle a,u \rangle$. We call this coprime pair \textit{trivial}. Notice that $\pair{0,1} \leq \pair{a,u}$.

\item If $e \in R$ is an idempotent, then $\pair{e,1-e}$ is a basic coprime pair which is easy to check that is minimal in $\RCP(R)$ with respect to $\leq$ (see \cite{GuilHerzog}).

\item Recall that, given $a \in R$, $aR$ is a direct summand of $R_R$ if and only if $a$ is a regular element. I.e., there exists an $x \in R$ with $axa=a$. If $\pair{a,b}$ is a right coprime pair, then $aR$ and $bR$ are direct summands if and only if $a$ and $b$ are regular elements. We call this kind of coprime pairs \textit{regular}. Notice that, in this case, we can find idempotent elements $e$ and $f$ such that $\pair{a,b}=\pair{e,f}$.
\end{enumerate}
\end{examples}

Let us prove some basic characterizations of right coprime pairs.
Recall that a left $R$ module $M$ satisfies the condition (C3) if for any two direct summands, $K$ and $L$, of $M$, $K+L$ is a direct summand provided that $K \cap L = 0$ (see e.g. \cite{MohamedMuller, NicholsonYousif}).

\begin{proposition} \label{p:CharacterizationCoprimePairs}
Let $a,b \in R$. The following assertions are equivalent:
\begin{enumerate}
\item $\pair{a,b}$ is a right coprime pair.

\item The morphism of left $R$-modules, $i:R \rightarrow R \oplus R$ given by $(x)i=(xa,xb)$, is a split monomorphism.

\item There exists a left $R$-module $M$ and two elements $m_1$ and $m_2$ of $M$ such that the morphism $f:R \rightarrow M$ defined by $(1)f = am_1+bm_2$ is a split monomorphism.

\item $\mathbf{l}_R(a) \cap \mathbf{l}_R(b)=0$ and $R(a,b)$ is a direct summand of $R\oplus R$.

\item $\mathbf{l}_R(a) \cap \mathbf{l}_R(b)=0$ and there exist $r,s,x_0,y_0 \in R$ such that $\mat{ra}{rb}{sa}{sb}$ is an idempotent matrix and 
\begin{displaymath}
(x_0,y_0)\mat{ra}{rb}{sa}{sb} = (a,b)
\end{displaymath}

\item $\mathbf{l}_R(a) \cap \mathbf{l}_R(b)=0$ and there exist $r,s \in R$ such that $\mathbf r_R(r) \cap \mathbf r_R(s) = 0$ and $\mat{ra}{rb}{sa}{sb}$ is an idempotent matrix.

\item $\pair{a+J(R),b+J(R)}$ is a right coprime pair in $R/J(R)$.
\end{enumerate}
If $R$ is von Neumann regular, or ${_R}R$ satisfies (C3) and $a$ and $b$ are idempotent elements, the preceding conditions are equivalent to:
\begin{enumerate}
\setcounter{enumi}{7}
\item $\mathbf l_R(a) \cap \mathbf l_R(b) = 0$.
\end{enumerate}
\end{proposition}

\begin{proof}
(1) $\Rightarrow$ (2). Suppose that $\pair{a,b}$ is a right coprime pair. Then the morphism $k:R \oplus R \rightarrow R$ given by $(x,y)k = xr+ys$ is a splitting of $i$, where $r$ and $s$ are scalars satisfying $1=ar+bs$.

(2) $\Rightarrow$ (3). Trivial.

(3) $\Rightarrow$ (1). Note that if $g$ splits $f$, then
\begin{displaymath}
1 = (1)fg = a(m_1)g+b(m_2)g,
\end{displaymath}
which implies that $\pair{a,b}$ is a right coprime pair.

(2) $\Leftrightarrow$ (4). The equivalence follows from the fact that $\Ker i = \mathbf l_R(a) \cap \mathbf l_R(b)$ and $\Img i = R(a,b)$.

(4) $\Rightarrow$ (5). If $R(a,b)$ is a direct summand of $R \oplus R$, there exists an idempotent endomorphism $h$ of $R \oplus R$ such that $\Img h = R(a,b)$. The endomorphism $h$ is of the form
\begin{displaymath}
(x,y)h=(x,y)\mat{u}{v}{w}{t}
\end{displaymath}
for some idempotent matrix $\mat{u}{v}{w}{t}$ with entries in $R$. Using that $(1,0)h$ and $(0,1)h$ belong to $R(a,b)$ we can find $r,s \in R$ such that $u=ra,v=rb,w=sa,t=sb$. Finally, since $(a,b)\in \Img h$, there exist $x_0,y_0 \in R$ such that
\begin{displaymath}
(x_0,y_0)\mat{ra}{rb}{sa}{sb} = (a,b)
\end{displaymath}

(5) $\Rightarrow$ (4). The endomorphism $h$ of $R \oplus R$ given by
\begin{displaymath}
(x,y)h=(x,y)\mat{ra}{rb}{sa}{sb}
\end{displaymath}
for all $x,y \in R$, is idempotent with $\Img h = R(a,b)$. Then, $R(a,b)$ is a direct summand of $R \oplus R$.

(1) $\Rightarrow$ (6). Since $i$ is injective, $0 = \Ker i = \mathbf l_R(a) \cap \mathbf l_R(b)$. Now take $r,s \in R$ such that $ar+bs=1$ and note that
\begin{displaymath}
\mat{ra}{rb}{sa}{sb}^2 = \mat{r(ar+bs)a}{r(ar+bs)b}{s(ar+bs)a}{s(ar+bs)b} = \mat{ra}{rb}{sa}{sb}
\end{displaymath}
Moreover, since $\pair{r,s}$ is a left coprime pair, it satisfies the dual of (2) and $j:R \rightarrow R \oplus R$ given by $j(x) = (rx,sx)$ is injective. This means that $\mathbf r_R(r) \cap \mathbf r_R(s) = 0$.

(6) $\Rightarrow$ (1). Using that the matrix is idempotent we obtain the identities
\begin{eqnarray*}
(ra)^2+rbsa=ra\\
rarb+rbsb=rb\\
sara+sbsa=sa\\
sarb+(sb)^2=sb
\end{eqnarray*}
The first two identities together with $\mathbf l_R(a) \cap \mathbf l_R(b) = 0$ give that $rar+rbs-r=0$. Similarly, from the last two identities we get that $sar+sbs-s=0$. Now, using that $\mathbf r_R(r) \cap \mathbf r_R(s) = 0$ we obtain that $ar+bs=1$. That is, $\pair{a,b}$ is a right coprime pair.

(1) $\Rightarrow$ (7). This is trivial. 

(7) $\Rightarrow$ (1). Note that (7) implies that $aR+bR+J(R)=R$. Then, by Nakayama's lemma, $aR+bR=R$ as well.

(4) $\Rightarrow$ (8). Trivial.

(8) $\Rightarrow$ (1). If $R$ is von Neumann regular, then $R(a,b)$ is always a direct summand of $R \oplus R$. Consequently, $\pair{a,b}$ is a right coprime pair by (4).
\medskip

Finally, if ${_R}R$ satisfies (C3) and $a$ and $b$ are idempotents satisfying (5), then $R(1-a) \cap R(1-b) = 0$. By condition (C3), $R(1-a)+R(1-b)$ is a direct summand of $R$. Therefore, $\pair{a,b}$ is a right coprime pair by \cite[Lemma 12]{ZimmermannZimmermann}.
\end{proof}



By a \textit{minimal coprime pair} we mean a right coprime pair which is minimal in $\RCP(R)$. We will use the following description of minimal coprime pairs (see \cite[Proposition 3]{GuilHerzog}).

\begin{proposition}\label{p:MinimalCoprimePairs}
Let $\pair{a,b}$ be a right coprime pair. The following assertions are equivalent:
\begin{enumerate}
\item The coprime pair $\pair{a,b}$ is minimal.

\item There exists an idempotent element $e$ such that $\pair{a,b}=\pair{e,1-e}$.

\item There exist $r,s \in R$ such that $a=ara$, $b=bsb$ (i. e., $\pair{a,b}$ is regular and $arbs=bsar=0$).

\end{enumerate}
\end{proposition}

\begin{proof}
In order to prove (1) $\Rightarrow$ (2), suppose that $\pair{a,b}$ is a minimal coprime pair and write $1=ar+bs$. Then $\pair{ar,bs} \leq \pair{a,b}$ and, since $\pair{a,b}$ is minimal, $\pair{ar,bs}$ is also minimal. By \cite[Proposition 3]{GuilHerzog} both $ar$ and $bs$ are idempotents. 

(2) $\Rightarrow$ (3). Write $e=ar$, $a=er'$, $1-e=bs$ and $b=(1-e)s'$. Then $a=ara$, $b=bsb$ and $arbs=bsar=0$.

(3) $\Rightarrow$ (1). By (3), $ar$ and $bs$ are orthogonal idempotents with $arR=aR$ and $bsR=bR$. Then $arR \oplus bsR=R$. This means that $\pair{ar,bs} = \pair{a,b}$ is a minimal right coprime pair.
\end{proof}

We obtain now some characterizations of rings $R$ in terms of the properties of their poset $\RCP(R)$. Recall from the introduction that $R$ is called an \textit{exchange ring} if $R$ has the finite exchange property as a right (equivalently, left) module, that is, for any module $X$ and decompositions
\begin{displaymath}
X = M \oplus N = A \oplus B
\end{displaymath}
with $M \cong R_R$, there exist submodules $A' \leq A$ and $B' \leq B$ such that $X=M \oplus A' \oplus B'$. Moreover, $R$ is said to be \textit{right quasi-duo} if every maximal right ideal is a left ideal. 

\begin{proposition}\label{p:CharacterizationRingsCoprimePairs} Let $R$ be a ring. Then:
\begin{enumerate}
\item $R_R$ is indecomposable if and only if $(\RCP(R),\leq)$ has exactly two minimal elements.

\item The following assertions are equivalent:
\begin{enumerate}
\item $R$ is local.
\item Every right coprime pair is trivial.

\item The following two conditions are satisfied:
\begin{enumerate}
\item $\RCP(R)$ has exactly two minimal elements.

\item For each $\pair{a,b} \in \RCP(R)$, there exists a minimal right coprime pair $\pair{c,d}$ such that $\pair{c,d} \leq \pair{a,b}$.
\end{enumerate}
\end{enumerate}

\item $R$ is left perfect if and only if $R$ has DCC on right coprime pairs.

\item $R$ is von Neumann regular if and only if every right coprime pair is regular.

\item $R$ is an exchange ring if and only if for every right coprime pair $\pair{a,b}$ there exists a minimal right coprime pair $\pair{c,d}$ with $\pair{c,d} \leq \pair{a,b}$.

\item $R$ is right quasi-duo if and only if every left coprime pair is a right coprime pair.
\end{enumerate}
\end{proposition}

\begin{proof}
(1). Note first that the right coprime pairs $\pair{1,0}$ and $\pair{0,1}$ are always minimal in $\RCP(R)$. 

Suppose now that $R_R$ is indecomposable and let $\pair{a,b} \in \RCP(R)$ be minimal. By Proposition \ref{p:MinimalCoprimePairs}, $aR \oplus bR=R$. Since $R$ is indecomposable, $aR=R$ or $bR=R$. In the first case, $\pair{1,0} = \pair{a,b}$. And, in the second, $\pair{0,1} = \pair{a,b}$.

Assume now that $\pair{1,0}$ and $\pair{0,1}$ are the only minimal elements of $\RCP(R)$ and let $e \in R$ be an idempotent. As  $\pair{e,1-e}$ is a minimal right coprime pair in $\RCP(R)$ by Proposition \ref{p:MinimalCoprimePairs}, we get that  $\pair{e,1-e}=\pair{1,0}$ or $\pair{e,1-e} = \pair{0,1}$, which means that $e \in \{0,1\}$. Thus, $R_R$ is indecomposable.

(2). (a) $\Leftrightarrow$ (b) follows from the fact that $R$ is local if and only if every proper right ideal of $R$ is superfluous. 

(b) $\Rightarrow$ (c). Trivial, since  the hypotheses imply that $\pair{0,1} \leq \pair{a,b}$ or $\pair{1,0} \leq \pair{a,b}$
for any right coprime pair $\pair{a,b}$. 

(c) $\Rightarrow$ (a). First, note that $\pair{1,0}$ and $\pair{0,1}$ are the only minimal elements in $\RCP(R)$. Choose two right ideals $I$ and $K$ of $R$ such that $I+K=R$. Then $1=y+k$ for some $y \in I$ and $k \in K$. By hypothesis, $\pair{1,0} \leq \pair{y,k}$ or $\pair{0,1} \leq \pair{y,k}$ which implies that $\pair{y,k}$ is trivial. This means that $I=R$ or $K=R$ and, consequently, that $R$ is local.

(3) If $\pair{a_0,b_0} \geq \pair{a_1,b_1} \geq \pair{a_2,b_2} \geq \cdots$ is a countable descending chain of right coprime pairs, then, since $R$ is left perfect, both chains $a_0R \geq a_1R \geq \cdots$ and $b_0R \geq b_1R \geq \cdots$ stabilize. In particular, there exist an $n < \omega$ such that $\pair{a_n,b_n} = \pair{a_{n+k}b_{n+k}}$ for each $k < \omega$.

Conversely, if $R$ has DCC on right coprime pairs and $a_0 R \geq a_1 R \geq a_2 R \geq \cdots$ is a descending chain of cyclic right ideals, then the chain of right coprime pairs
\begin{displaymath}
\pair{a_0,1} \geq \pair{a_1,1} \geq \pair{a_2,1} \geq \cdots
\end{displaymath}
stabilizes, which implies that there exists an $n < \omega$ such that $a_nR = a_{n+k}R$ for each $k < \omega$.

(4) Trivial.

(5) Suppose that $R$ is an exchange ring and let $\pair{a,b}$ be a right coprime pair. Writing $1=ar+bs$ for some $r,s \in R$, we get that $\pair{ar,1-ar} \leq \pair{a,b}$. Now, by \cite[Theorem 2.1]{Nicholson}, there exists an idempotent $e$ with $eR \leq arR$ and $(1-e)R \leq (1-ar)R$. By Proposition \ref{p:MinimalCoprimePairs}, $\pair{e,1-e}$ is a minimal right coprime pair which satisfies $\pair{e,1-e} \leq \pair{a,b}$.

Conversely, given $x \in R$, there exists a minimal right coprime pair $\pair{c,d}$ with $\pair{c,d} \leq \pair{x,1-x}$ . By Proposition \ref{p:MinimalCoprimePairs}, there exists an idempotent $e$ such that $\pair{c,d} = \pair{e,1-e}$. Again by \cite[Theorem 2.1]{Nicholson}, this implies that $R$ is an exchange ring.

(6) This is proved in \cite[Theorem 3.2]{LamDugas}.
\end{proof}

\begin{remark}
Note that $R$ is local if and only if $R$ is exchange and $\RCP(R)$ has exactly two minimal elements.
\end{remark}

\section{Strongly exchange rings. Examples}

As we mentioned in the introduction, one of the key ingredients in \cite{GuilHerzog} is that certain descending chains of right coprime pairs have lower bounds. In this section, we give the central notion of compatible descending chain of right coprime pairs, and study the lower bounds of them.

\begin{definition}\label{d:CompatibleSystem}
A compatible descending chain of right coprime pairs is a chain of right coprime pairs, $\{p_\alpha\mid \alpha < \kappa\}$, such that, for each $\alpha < \kappa$, there exists a pair of generators $(a_\alpha,b_\alpha)$ of $p_\alpha$, and families of scalars $\{r_{\alpha\beta}\mid \alpha < \beta\}$ and $\{s_{\alpha\beta} \mid \alpha < \beta\}$, satisfying the following two conditions  for each ordinals $\alpha < \gamma < \beta$ with $\beta<\kappa$.
\begin{enumerate}
\item $a_\beta=a_\alpha r_{\alpha\beta}$ and $b_\beta=b_\alpha s_{\alpha\beta}$;

\item $r_{\alpha\beta}=r_{\alpha\gamma}r_{\gamma\beta}$ and $s_{\alpha\beta}=s_{\alpha\gamma}s_{\gamma\beta}$.
\end{enumerate}
\end{definition}

From now on, if $\{\pair{a_\alpha,b_\alpha}\mid \alpha<\kappa\}$ is a compatible descending chain of right coprime pairs, we will assume that $(a_\alpha,b_\alpha)$ is the pair of generators satisfying the compatibility condition of the preceding definition.

Of course, we could have defined descending chains with the index set being a totally ordered set instead of an ordinal. However, since every totally ordered set contains a cofinal well ordered subset (see e.g.\cite[Theorem 36]{Roitman}), we can always assume that the index set of the chain is a well ordered set.

We are interested in studying when these chains have lower bounds in $\RCP(R)$. A \textit{minimal lower bound} of a chain is a lower bound of the chain that is a minimal element in $\RCP(R)$.

\begin{proposition}\label{p:PropertiesChainsCoprimePairs}
Let $\{\pair{a_\alpha,b_\alpha} \mid \alpha < \kappa\}$ be a descending chain of right coprime pairs. Then:
\begin{enumerate}
\item The chain has a lower bound if and only if
\begin{displaymath}
\bigcap_{\alpha < \kappa} a_\alpha R + \bigcap_{\alpha < \kappa}b_\alpha R = R
\end{displaymath}

\item If $\pair{a_\alpha,b_\alpha}$ is regular for each $\alpha < \kappa$, then the chain is compatible. In particular, every descending chain of right coprime pairs in a von Neumann regular ring is compatible.

\item If $R$ is von Neumann regular, the following are equivalent:
\begin{enumerate}
\item The chain has a lower bound.

\item There exist $a,b\in R$ with $\mathbf l_R(a) \cap \mathbf l_R(b) = 0$, $\sum_{\alpha < \kappa}\mathbf l_R(a_\alpha) \leq \mathbf l_R(a)$ and $\sum_{\alpha < \kappa} \mathbf l_R(b_\alpha) \leq \mathbf l_R(b)$.
\end{enumerate}
\end{enumerate}
\end{proposition}

\begin{proof}
(1) If the chain has a lower bound, then trivially $\bigcap_{\alpha < \kappa}a_\alpha R+\bigcap_{\alpha < \kappa}b_\alpha R=R$. Conversely, write $1=a+b$ with $a \in \bigcap_{\alpha < \kappa}a_\alpha R$ and $b \in \bigcap_{\alpha < \kappa}b_\alpha R$. Then $\pair{a,b}$ is a right coprime pair and a lower bound of the system.

(2) Since $\pair{a_\alpha,b_\alpha}$ is regular, we may assume that $a_\alpha$ and $b_\alpha$ are idempotent elements for each $\alpha < \kappa$. Given $\alpha < \beta$, we have that $a_{\alpha} a_{\beta} = a_{\beta}$, since $a_{\beta} = a_\alpha r$ for some $r \in R$. Then the families of scalars $\{r_{\alpha\beta}\mid \alpha < \beta\}$ and $\{s_{\alpha\beta}\mid \alpha<\beta\}$, given by $r_{\alpha\beta} = a_{\beta}$ and $s_{\alpha\beta}=b_\beta$, for each pair $\alpha<\beta$, make the chain compatible. 

(3) Again, we may assume that $a_\alpha$ and $b_\alpha$ are idempotent elements for each $\alpha$.

(a) $\Rightarrow$ (b). Take $\pair{a,b}$ a lower bound of the chain. Then $\mathbf l_R(a) \cap \mathbf l_R(b) = 0$ by Proposition \ref{p:CharacterizationCoprimePairs}. Moreover, the inclusions
\begin{displaymath}
\bigcap_{\alpha < \kappa}a_\alpha R \geq aR \textrm{ and }\bigcap_{\alpha < \kappa}b_\alpha R \geq bR
\end{displaymath}
imply, by \cite[Propositions 2.15 and 2.16]{AndersonFuller}, that
\begin{displaymath}
\sum_{\alpha < \kappa}\mathbf l_R(a_\alpha R) \leq \mathbf l_R\left(\bigcap_{\alpha < \kappa} a_\alpha R\right) \leq \mathbf l_R(a) \textrm{ and } \sum_{\alpha < \kappa}\mathbf l_R(b_\alpha R) \leq \mathbf l_R\left(\bigcap_{\alpha < \kappa}b_\alpha R\right) \leq \mathbf l_R(b).
\end{displaymath}

(b) $\Rightarrow$ (a). Since $R$ is von Neumann regular, we may assume that $a$ and $b$ are idempotents. By Proposition \ref{p:CharacterizationCoprimePairs}, $\pair{a,b}$ is a right coprime pair. Moreover, since $aR = \mathbf r_R \mathbf l_R(a)$ and $a_\alpha R = \mathbf r_R \mathbf l_R (a_\alpha R)$ for each $\alpha < \kappa$, we get, by \cite[Propositions 2.15 and 2.16]{AndersonFuller} that
\begin{displaymath}
aR = \mathbf r_R \mathbf l_R(a) \leq \mathbf r_R \left(\sum_{\alpha < \kappa}\mathbf l_R(a_\alpha R)\right) = \bigcap_{\alpha < \kappa} \mathbf r_R\mathbf l_R(a_\alpha) = \bigcap_{\alpha < \kappa}a_\alpha R.
\end{displaymath}

And, similarly,
\begin{displaymath}
bR \leq \bigcap_{\alpha < \kappa}b_\alpha R.
\end{displaymath}
This means that $\pair{a,b}$ is a lower bound of the chain.
\end{proof}

Motivated by the preceding results, we introduce the following natural notion of strongly exchange ring:

\begin{definition}
We say that a ring $R$ is right strongly exchange if every compatible descending chain of right coprime pairs has a minimal lower bound.
\end{definition}

\begin{examples}
\begin{enumerate}
\item Let $R$ be an integral domain. Then, for every pair of elements $a$ and $b$ such that $a \in bR$, there exists a unique $r \in R$ such that $a=br$. As a consequence, every descending chain of right coprime pairs is compatible.

\item $\mathbb Z$ is not a strongly exchange ring since, for instance, if $p$ and $q$ are different primes, the compatible descending chain $\{\pair{p^n,q^n}\mid n < \omega\}$ does not have a lower bound.
\end{enumerate}
\end{examples}

We show an example of an exchange ring which is not right strongly exchange. Actually, we are going to construct, for any regular cardinal $\kappa$, a von Neumann regular ring $S$ such that all compatible descending chains of right coprime pairs of cardinality smaller than $\kappa$ have minimal lower bounds, but there does exist a chain of length $\kappa$ with no lower bound.

Recall that a cardinal $\kappa$ is called \textit{singular} if there exists a cardinal $\mu < \kappa$ and a family of cardinals $\{\kappa_\alpha \mid \alpha < \mu\}$ with $\kappa_\alpha < \kappa$ for each $\alpha$, such that $\kappa = \sum_{\alpha < \mu}\kappa_\alpha$. A cardinal is called \textit{regular} when it is not singular.

\begin{example}\label{e:RegularRingNotDescendingChain}
Let $\kappa$ be an infinite regular cardinal. There exists a von Neumann regular ring $S$ satisfying that:
\begin{enumerate}
\item Every compatible descending chain  of right coprime pairs of length smaller than $\kappa$ has a minimal lower bound.

\item There exists a compatible descending chain of right coprime pairs with no lower bound.
\end{enumerate}
Since every von Neumann regular ring is an exchange ring, this example shows that exchange rings do not need to be strongly exchange.
\end{example}

\begin{proof}
Let $F$ a field and $S$, the subring of $F^{\kappa}$ given by
\begin{displaymath}
\{x \in F^{\kappa}\mid \exists C \subseteq \kappa \textrm{ with }|C|<\kappa \textrm{ and }x(\alpha) = x(\beta) \, \forall \alpha,\beta \in \kappa\setminus C\}.
\end{displaymath}

It is clear that $S$ is von Neumann regular, since for any $x \in S$, $x$ is of the form $xyx$, where, for each $\alpha$, $y(\alpha) = x(\alpha)^{-1}$ if $x(\alpha)\neq 0$, and $y(\alpha)=0$ otherwise.

(1) First note that $x \in S$ is idempotent if and only if $x(\alpha) \in \{0,1\}$ for each $\alpha < \kappa$ and one of the sets, $\{\alpha < \kappa \mid x(\alpha)=1\}$ or $\{\alpha < \kappa \mid x(\alpha)=0\}$, has cardinality smaller than $\kappa$. Indeed, if $\Gamma$ is any subset of $\kappa$ and we denote by $e_{\Gamma}$ the element of $F^\kappa$ satisfying that
\begin{displaymath}
e_{\Gamma}(\alpha) = \left\{\begin{array}{ll}1 & \textrm{if $\alpha \in \Gamma$}\cr 0 & \textrm{if $\alpha \notin \Gamma$}\cr
\end{array}\right.
\end{displaymath}
then the set of idempotents of $S$ is $\{e_\Gamma \mid \Gamma \subseteq \kappa \textrm{ and } |\Gamma| < \kappa \textrm{ or } |\kappa\setminus\Gamma| < \kappa\}$.

Let us choose an ordinal $\lambda < \kappa$ and a descending chain of right coprime pairs $\{\pair{a_\alpha,b_\alpha}\mid \alpha < \lambda\}$. We may assume that both $a_\alpha$ and $b_\alpha$ are idempotents. Since $\pair{a_\alpha,b_\alpha}$ is a right coprime pair, 
\begin{equation}\label{e:UnionSupports}
\supp(a_\alpha) \cup \supp(b_\alpha)= \kappa,
\end{equation}
for each $\alpha < \lambda$.

Let us denote by $\Gamma = \bigcap_{\alpha < \kappa}\supp(a_\alpha)$ and $\Delta = \bigcap_{\alpha < \kappa} \supp(b_\alpha)$. We claim that
\begin{enumerate}
\item[(A)] $\Gamma \cup \Delta = \kappa$.

\item[(B)] Either $\Gamma$ or $\kappa\setminus\Gamma$ has cardinality smaller than $\kappa$. And, similarly, either $\Delta$ or $\kappa\setminus\Delta$ has cardinality smaller than $\kappa$.
\end{enumerate}
Assume that this claim is already proved and choose the idempotents $e_\Gamma$ and $e_\Delta$. They belong to $S$ by (B). And $\pair{e_\Gamma,e_\Delta}$ is a right coprime pair by (A), which is clearly a lower bound of the initial chain. Therefore, $\pair{e_\Gamma,e_{\Delta\setminus\Gamma}}$ is a minimal lower bound of the chain.

Let us prove our claim. In order to prove (A), take $x \in \kappa$. Then $x \in \supp(a_0)$ or $x \in \supp(b_0)$ by equation (1). Suppose $x \in \supp(a_0)$. If $x \notin \supp(b_0)$, then $x \in \Gamma$ by (1) and we are done. If $x \in \supp(b_0)$, we have two possibilities. If $x \in \supp(a_\alpha) \cap \supp(b_\alpha)$ for each $\alpha$, then $x \in \Gamma \cap \Delta$ and we are done. Otherwise, take $\alpha$ the minimum ordinal such that $x \notin \supp(a_\alpha) \cap \supp(b_\alpha)$ and assume, by (1), that $x \in \supp(a_\alpha)$. Then, again by (1), $x \in \Gamma$. This proves claim (A).

In order to prove (B), let us  we check that either $\Gamma$ or $\kappa\setminus\Gamma$ has cardinality smaller than $\kappa$. The proof involving $\Delta$ is similar. If there exists an $\alpha < \kappa$ such that $|\supp(a_\alpha)|<\kappa$, then $\Gamma \subseteq \supp(a_\alpha)$ has cardinality smaller than $\kappa$. If $|\supp(a_\alpha)|=\kappa$ for each $\alpha < \lambda$, then $|\kappa\setminus\supp(a_\alpha)|<\kappa$. Therefore,
\begin{displaymath}
|\kappa \setminus \Gamma| \leq \sum_{\alpha < \lambda}|\kappa \setminus \supp(a_\alpha)|<\kappa
\end{displaymath}
since $\kappa$ is regular and $\lambda < \kappa$.

(2) Let $I_1$ and $I_2$ be two subsets of $\kappa$ such that $|I_1|=|I_2|=\kappa$, $I_1 \cap I_2 = \emptyset$ and $\kappa=I_1 \cup I_2$. Consider in $I_1$ and $I_2$ the orders induced by $\kappa$ and set, for each $\alpha <\kappa$,  $\Gamma_\alpha = \kappa\setminus\{\beta\in I_1\mid \beta < \alpha\}$ and $\Delta_\alpha=\kappa\setminus\{\gamma\in I_2 \mid \gamma < \alpha\}$. We obtain a descending chain of right coprime pairs $\{\pair{e_{\Gamma_\alpha},e_{\Delta_\alpha}}\mid \alpha < \kappa\}$. Let us check that this chain does not have a lower bound. Suppose that $\pair{a,b}$ is a lower bound. Then

\begin{enumerate}[label=(\roman*)]
\item $\supp(a) \subseteq \bigcap_{\alpha < \kappa}\supp(e_{\Gamma_\alpha})$,

\item $\supp(b) \subseteq \bigcap_{\alpha < \kappa}\supp(e_{\Delta_\alpha})$ and

\item $\supp(a) \cup \supp(b) = \kappa$.
\end{enumerate}

Since $I_1 \cap \left( \bigcap_{\alpha < \kappa}\supp(e_{\Gamma_\alpha})\right)=\emptyset$, (i) implies that $|\{\alpha < \kappa \mid a(\alpha)=0\}|=\kappa$. Since $I_2 \cap \left(\bigcap_{\alpha < \kappa}\supp(e_{\Delta_\alpha})\right)=\emptyset$, (ii) implies that $I_2 \cap \supp(b)=\emptyset$, so that $I_2 \subseteq \supp(a)$ by (iii). Then, $|\supp(a)|=\kappa$.

We have proved that $|\supp(a)|=|\{\alpha < \kappa\mid a(\alpha)=0\}|=\kappa$, which contradicts that $a \in R$.
\end{proof}


Now we show, using ideas from \cite[Lemma 2]{GuilHerzog}, that left cotorsion rings are right strongly exchange. Recall that $R$ is left cotorsion if $\Ext_R^1(F,R)=0$ for each flat left $R$-module $F$ (see e.g. \cite{GuilHerzog}). Let us begin by discussing the relationship between compatible descending chains of right coprime pairs and certain direct limits.

Let $\{\pair{a_\alpha,b_\alpha} \mid 
\alpha < \kappa\}$ be a compatible descending chain of right coprime pairs, and choose families of scalars $\{r_{\alpha\beta}\mid \alpha < \beta\}$ and $\{s_{\alpha\beta} \mid \alpha<\beta\}$ witnessing the compatibility of the chain. In a similar way as in \cite[Lemma 2]{GuilHerzog}, we can construct a direct system of split exact sequences $(\Sigma_\alpha,(f_{\alpha\beta},g_{\alpha\beta},h_{\alpha\beta}) \mid \alpha < \beta)$ in $\Modl R$, associated to this descending chain and the families of scalars, as follows. For each $\alpha$, $\Sigma_\alpha$ is the short exact sequence
\begin{displaymath}
\begin{tikzcd}
0 \arrow{r} & R \arrow{r}{k_{\alpha}} & R \oplus R \arrow{r}{p_\alpha} & Z_\alpha \arrow{r} & 0,
\end{tikzcd}
\end{displaymath}
where $(x)k_\alpha = (xa_\alpha,xb_\alpha)$ for each $x \in R$, and $p_\alpha$ is a cokernel of $k_\alpha$. The fact that $\pair{a_\alpha,b_\alpha}$ is a right coprime pair is equivalent to $\Sigma_\alpha$ being split. 

The morphisms $f_{\alpha\beta}$, $g_{\alpha\beta}$ and $h_{\alpha\beta}$ are defined in the following way:
\begin{itemize}
\item $f_{\alpha\beta}$ is the identity;

\item $(x,y)g_{\alpha\beta} = (xr_{\alpha\beta},ys_{\alpha\beta})$;

\item $h_{\alpha\beta}$ is the unique morphism from $Z_\alpha$ to $Z_\beta$ induced by $g_{\alpha\beta}$.
\end{itemize}

Note that, as the tensor product commutes with direct limits, the direct limit of the system of sequences is of the form
\begin{displaymath}
\begin{tikzcd}
0 \arrow{r} & R \arrow{r}{k} & F_1 \oplus F_2 \arrow{r}{p} & Z \arrow{r} & 0
\end{tikzcd}
\end{displaymath}
with $F_1$ and $F_2$ flat modules.

\begin{theorem}\label{t:CotorsionAreStrongly}
Let $R$ be a left cotorsion ring. Then $R$ is right strongly exchange.
\end{theorem}

\begin{proof}
Suppose that the result is not true, that is, that $R$ is not right strongly exchange. Then there exists a compatible descending chain of right coprime pairs, $\{\pair{a_\alpha,b_\alpha}\mid \alpha < \kappa\}$, which does not have a minimal lower bound.

For each ordinal $\alpha$, we are going to construct a right coprime pair $\pair{x_\alpha,y_\alpha}$ such that $\pair{x_\alpha,y_\alpha} < \pair{x_\gamma,y_\gamma}$ for each $\gamma < \alpha$ and the chain $\{\pair{x_\gamma,y_\gamma}\mid \gamma \leq \alpha\}$ is compatible. This is a contradiction, since it implies that the cardinality of $\RCP(R)$ is bigger than the cardinality of $\alpha$ for any ordinal $\alpha$.

We will make our construction by transfinite induction on all ordinals $\alpha$. If $\alpha < \kappa$, set $x_\alpha = a_\alpha$ and $y_\alpha = b_\alpha$. Let us choose now an ordinal $\alpha \geq \kappa$ and assume that we have constructed right coprime pairs $\pair{x_\gamma,y_\gamma}$ for each $\gamma < \alpha$ satisfying the above condition. Let us distinguish among two possibilities.

Suppose first that $\alpha$ is a successor ordinal, say $\alpha = \beta+1$. As $\pair{x_\beta,y_\beta}$ is not minimal, by the election of the initial chain, there exists a right coprime pair $\pair{x_{\alpha},y_{\alpha}} \in \RCP(R)$ strictly smaller than $\pair{x_\beta,y_\beta}$. Clearly, the chain $\{\pair{x_\gamma,y_\gamma} \mid \gamma \leq \alpha\}$ is compatible.

Suppose now that $\alpha$ is a limit ordinal. Choose families of scalars, $\{r_{\alpha\beta}\mid \alpha < \beta\}$ and $\{s_{\alpha\beta} \mid \alpha<\beta\}$, witnessing the compatibility of the chain, and consider the direct system of splitting short exact sequences, $(\Sigma_\gamma,(f_{\gamma \beta},g_{\gamma \beta},h_{\gamma \beta}) \mid \gamma < \beta < \alpha)$, associated to this descending chain. As shown before, the direct limit of the system is of the form
\begin{equation}\label{e:Sequence}
\begin{tikzcd}
0 \arrow{r} & R \arrow{r}{k} & F_1 \oplus F_2 \arrow{r}{p} & Z \arrow{r} & 0
\end{tikzcd}
\end{equation}
with $F_1$ and $F_2$ flat modules. This sequence is pure, since all the sequences in the system are split (see e.g. \cite[34.5]{Wisbauer}). This implies that $Z$ is flat by \cite[36.6]{Wisbauer}. We deduce now that this sequence is split, since $R$ is left cotorsion.

Take a splitting $k'$ of $k$ and denote by $f_\gamma:R \oplus R \rightarrow F_1 \oplus F_2$ the canonical map associated to the direct limit, for each $\gamma < \alpha$. Write $(1)k = z_1+z_2$ for some $z_1 \in F_1$ and $z_2 \in F_2$ and note that $(x_\gamma,0)f_\gamma = z_1$ and $(0,y_\gamma)f_\gamma = z_2$ for each $\gamma < \alpha$. Denote by $x_\alpha = (z_1)k'$, $y_\alpha = (z_2)k'$, $r_{\gamma} = (1,0)f_\gamma k'$ and $s_{\gamma} = (0,1)f_\gamma k'$ for each $\gamma < \alpha$. Then $\pair{x_\alpha,y_\alpha}$ is a right coprime pair since
\begin{displaymath}
1 = (1)kk' = (z_1+z_2)k' = x_\alpha+y_\alpha.
\end{displaymath}
Moreover, for each $\gamma < \alpha$,
\begin{displaymath}
x_\gamma r_{\gamma} = (x_\gamma,0)f_\gamma k' = x_\alpha \textrm{ and }y_\gamma s_{\gamma} = (0,y_\gamma)f_\gamma k' = y_\alpha.
\end{displaymath}
Finally, if $\gamma<\beta$,
\begin{displaymath}
r_{\gamma\beta}r_{\beta} = (1,0)f_{\gamma\beta}f_\beta k' = (1,0)f_\gamma k' = r_{\gamma}
\end{displaymath}
and
\begin{displaymath}
s_{\gamma\beta}r_{\beta} = (0,1)f_{\gamma\beta}f_\beta k' = (0,1)f_\gamma k' = s_{\gamma}.
\end{displaymath}

These identities mean that $\pair{x_\alpha,y_\alpha}$ is lower bound of the chain and that $\{\pair{x_\gamma,y_\gamma}\mid \gamma \leq \alpha\}$ is a compatible descending chain of right coprime pairs. Moreover, $\pair{x_\alpha, y_\alpha} < \pair{x_\gamma,y_\gamma}$ for each $\gamma < \alpha$, since the equality $\pair{x_\alpha,y_\alpha}=\pair{x_\gamma,y_\gamma}$ for some $\gamma < \alpha$ implies that $\pair{x_\gamma,y_\gamma}=\pair{x_{\gamma+1},y_{\gamma+1}}$, which is a contradiction. This finishes the construction.
\end{proof}

Now we prove that local rings are right strongly exchange.

\begin{proposition}\label{p:LocalAreStrongly}
Every local ring is right (and left) strongly exchange.
\end{proposition}

\begin{proof}
Take a compatible descending chain of right coprime pairs, $\{\pair{a_\alpha,b_\alpha}\mid \alpha<\kappa\}$. If $\pair{a_\alpha,b_\alpha}=\pair{1,1}$ for each $\alpha<\kappa$, then $\pair{1,0}$ is trivially a minimal lower bound of the chain.

So we may assume that there exists an $\alpha < \kappa$ such that $\pair{a_\alpha,b_\alpha} \neq \pair{1,1}$. Since $\pair{a_\alpha,b_\alpha}$ is a trivial coprime pair  by Proposition \ref{p:CharacterizationRingsCoprimePairs}, either $a_\alpha$ or $b_\alpha$ is a unit. Suppose, without loss of generality, that $a_\alpha$ is a unit. Then $b_\alpha \in J(R)$. We claim that $a_\beta$ is a unit for each $\beta < \kappa$. Suppose, on the contrary, that there exists a $\beta<\kappa$ such that $a_\beta$ is not a unit. Then $a_\beta \in J(R)$. Choosing $\gamma={\rm max\{\alpha,\beta\}}$, we deduce that $a_\gamma\in J(R)$ (since $a_\gamma\in a_\beta R$) and $b_\gamma \in J(R)$ (since $b_\gamma\in b_\alpha R$). But then, $\pair{a_\gamma,b_\gamma}$ cannot be a right coprime pair. A contradiction that proves our claim.

Consequently, $a_\alpha$ is a unit for each $\alpha<\kappa$ and thus, $\pair{1,0}$ is a minimal lower bound of the chain.
\end{proof}

Using this result we can show that the class of left cotorsion rings is strictly contained in the class of right strongly exchange rings.

\begin{example}
Let $R$ be a local Noetherian commutative ring which is not complete in the $I$-adic topology. As the completion in the $I$-adic topology of $R$ is the pure-injective envelope of $R$ (see e.g. \cite[Example 7.7]{JensenLenzing}), this means that $R$ cannot be cotorsion by \cite[Lemma 3.1.6]{Xu}. However, $R$ is strongly exchange by Proposition \ref{p:LocalAreStrongly}.
\end{example}

We close this section by showing that the endomorphism ring of a continuous left $R$-module is right strongly exchange. Recall that a left $R$-module $M$ is called continuous when it satisfies the following two conditions (see e.g. \cite{MohamedMuller, NicholsonYousif}):
\begin{enumerate}
\item[(C1)] Every submodule of $M$ is essential in a direct summand.

\item[(C2)] Any submodule of $M$ which is isomorphic to a direct summand of $M$ is itself a direct summand.
\end{enumerate}

For the rest of this section, let us fix a left $R$-module $M$ and an injective envelope $u:M \rightarrow E$ of $M$. We will denote by $T$ and $S$ the endomorphism rings of $M$ and $E$, respectively. For every $f \in T$, we can use the injectivity of $E$ to find an extension $\hat f$ of $f$ to $E$, that is, a morphism $\hat f:E \rightarrow E$ such that $u\hat f=fu$.

Let $f,g \in T$ and take two extensions, $\hat f$ and $\hat g$, of these morphisms to $E$. In general, it may happen that $\hat f S \leq \hat g S$ but $fT \nleq gT$. Our next lemma shows that this is not the case when $M$ is continuous and $f$ is idempotent.

\begin{lemma}\label{l:CyclicContinuous}
Let $M$ be a continuous left $R$-module and $f,e\in T$ with $e$ idempotent. Take $\hat f$ and $\hat e$ extensions of $f$ and $e$ to $E$, respectively. If $\hat e S \leq \hat f S$, then $eT \leq fT$.
\end{lemma}

\begin{proof}
There exists, by hypothesis, an $h \in S$ such that $\hat e=\hat fh$.  The restriction of $f$ to $\Img e$ is a monomorphism from $\Img e$ to $M$ with image $\Img ef$, since if $x \in M$ satisfies $(x)ef=0$, then $0 = (x)\hat e \hat f h=(x)\hat e \hat e = (x)ee = (x)e$. We can apply now (C2) to conclude that $\Img ef$ is a direct summand of $M$.

On the other hand, by condition (C1), we can find a direct summand $K$ of $M$ such that $\Img (1-e)f$ is essential in $K$. We claim that $\Img ef \cap K = 0$. Suppose, on the contrary, that there exists $0\neq y \in  \Img ef \cap K$. Write $y=(x)ef$ for some $x \in M$, and note that, as $\Img(1-e)f$ is essential in $K$, there exist $z \in M$ and $r \in R$ such that $0\neq (rx)ef=(z)(1-e)f$. Applying $h$ to this identity we get, on one hand, that $(rx)efh=(rx)e \neq 0$ and, on the other, that $(z)(1-e)fh=(z)(1-e)e=0$, a contradiction that proves our claim.

We can now apply (C3) to conclude that $\Img ef\oplus K$ is a direct summand of $M$. Thus, there exists a submodule $L$ of $M$ such that $\Img ef \oplus K \oplus L = M$. Consider the endomorphism $h'$ of $M$ whose restriction to $\Img ef$ is equal to $h$ (note that, $\Img efh \leq M$), and its restriction to $K\oplus L$ is zero. Then, for each $x \in M$,
\begin{displaymath}
(x)fh'=((x)e+(x)(1-e))fh'=(x)efh'+(x)(1-e)fh'=(x)e.
\end{displaymath}
This means that $fh'=e$ and that $eT \leq fT$.
\end{proof}

The following lemma allows us to construct right coprime pairs in $S$ from right coprime pairs in $T$.

\begin{lemma}\label{l:ContorsionPairEndomorphismRingInjectiveHull}
Let $\pair{f,g}$ be a right coprime pair in $T$ and $\hat f$ and $\hat g$ extensions of $f$ and $g$ to $E$, respectively. Then $\pair{\hat f,\hat g}$ is a right coprime pair in $S$.
\end{lemma}

\begin{proof}
Take $r,s \in T$ such that $fr+gs=1_M$. Then $u(\hat f \hat r+\hat g \hat s )=u$. As $u$ is an injective envelope, $\hat f \hat r + \hat g \hat s$ is an isomorphism and thus, there exists an $h \in S$ such that $1_E = \hat f \hat rh+\hat g \hat sh$. This means that $\pair{\hat f,\hat g}$ is a right coprime pair in $S$.
\end{proof}
\begin{theorem}\label{t:ContinuousHaveCompatibleLowerBounds}
Let $M$ be a continuous left $R$-module. Then $\End_R(M)$ is right strongly exchange.
\end{theorem}

\begin{proof}
We follow the notation fixed in p. 11. Take a compatible descending chain of right coprime pairs in $T$, $\{\pair{f_\alpha,g_\alpha}\mid \alpha < \kappa\}$, and families of elements of $T$, $\{r_{\alpha\beta}\mid \alpha < \beta\}$ and $\{s_{\alpha\beta}\mid \alpha < \beta\}$, making the chain compatible. First, we are going to construct families of endomorphisms of $E$, $\{\hat r_{\alpha\beta}\mid \alpha < \beta\}$ and $\{\hat s_{\alpha\beta}\mid \alpha < \beta\}$, extending the families $r_{\alpha\beta}$ and $s_{\alpha\beta}$, respectively, to $E$, such that
\begin{enumerate}
\item[(A)] $\hat r_{\alpha\gamma}\hat r_{\gamma \beta} = \hat r_{\alpha\beta}$ and

\item[(B)] $\hat s_{\alpha\gamma}\hat s_{\gamma \beta} = \hat s_{\alpha\beta}$,
\end{enumerate}
for each $\alpha < \gamma < \beta$. We make the construction of $\{\hat r_{\alpha\beta}\mid \alpha < \beta\}$ by transfinite induction on $\beta$. The construction of the family $\{\hat s_{\alpha\beta}\mid \alpha < \beta\}$ is made similarly. 

If $\beta=0$, there is nothing to construct. If $\beta=1$, choose, using the injectivity of $E$, an extension $\hat r_{01}:E\rightarrow E$ of $r_{01}$.

Let $\beta < \kappa$, and suppose that we have constructed $\hat r_{\alpha \gamma}$ for each $\alpha < \gamma < \beta$, and let us construct $\hat r_{\alpha \beta}$. If $\beta$ is successor, say $\beta = \gamma+1$, take $\hat r_{\gamma\gamma+1}\in S$ an extension of $r_{\gamma\gamma+1}$ and define $\hat r_{\alpha\beta}=\hat r_{\alpha\gamma}\hat r_{\gamma\beta}$. It is easy to check that $\hat r_{\alpha\beta}$ is an extension of $r_{\alpha\beta}$ and that the compatibility conditions (A) hold.

Suppose now that $\beta$ is a limit ordinal. We have two direct systems of left $R$-modules: ${\mathcal S}_1=(M_\alpha, r_{\alpha\gamma}\mid \alpha < \gamma \in \beta)$ and ${\mathcal S_2}=(E_\alpha, \hat r_{\alpha\gamma} \mid \alpha < \gamma \in \beta)$, where $M_\alpha=M$ and $E_\alpha = E$ for each $\alpha < \beta$. Denote by $(X,m_\alpha:M\rightarrow X \mid \alpha < \beta)$ and $(Y,n_\alpha:E\rightarrow Y \mid \alpha < \beta)$ their direct limits. Since each $\hat r_{\alpha \gamma}$ is an extension of $r_{\alpha\gamma}$, $u$ defines a monomorphism between the direct systems $\mathcal S_1$ and $\mathcal S_2$, which induces a monomorphism $v:X \rightarrow Y$. Similarly, $( r_{\alpha\beta}:M_\alpha \rightarrow M \mid \alpha < \beta)$ is a direct system of morphisms from $\mathcal S_1$ to $M$. So, it induces a morphism $r:X \rightarrow M$. Finally, using the injectivity of $E$ we can find a $\hat r:Y \rightarrow E$ making the diagram
\begin{displaymath}
\begin{tikzcd}
X \arrow{r}{v} \arrow{d}{r} & Y \arrow[dotted]{d}{\hat r}\\
M \arrow{r}{u} & E
\end{tikzcd}
\end{displaymath}
commutative. Set now $\hat r_{\alpha \beta} = n_\alpha \hat r$ for each $\alpha < \beta$. It is easy to check that $\hat r_{\alpha \beta}$ is an extension of $r_{\alpha\beta}$ that satisfies the compatibility conditions (A). This finishes the construction.

Take now $k_0$ and $l_0$ extensions of $f_0$ and $g_0$ to endomorphisms of $E$, respectively, and define
\begin{displaymath}
k_\alpha = k_0\hat r_{0\alpha} \textrm{ and } l_\alpha =  l_0\hat s_{0\alpha}
\end{displaymath}
for every $\alpha < \kappa$. Since $k_\alpha$ and $l_\alpha$ are extensions of $f_{\alpha}$ and $g_{\alpha}$ to $E$, respectively, $\pair{k_\alpha,l_\alpha}$ is a right coprime pair in $S$ by Lemma \ref{l:ContorsionPairEndomorphismRingInjectiveHull}. We obtain in this way a compatible descending chain $\{\pair{k_\alpha,l_\alpha} \mid \alpha < \kappa\}$ of right coprime pairs in $S$. Since $S$ is left pure-injective (in particular, left cotorsion) by \cite[Proposition 3]{ZimmermannZimmermann}, the chain has a minimal lower bound $\pair{k,l}$ by Theorem \ref{t:CotorsionAreStrongly}. So, there exists an idempotent $e$ of $S$ such that $\pair{k,l}=\pair{e,1-e}$ by Proposition \ref{p:MinimalCoprimePairs}. By \cite[Theorem 2.8]{MohamedMuller}, there exists an idempotent $e' \in T$ such that $e$ is an extension of $e'$ to $E$ and, consequently, $1-e$ is also an extension of $1-e'$ to $E$. Since $\pair{e,1-e} \leq \pair{k_\alpha,l_\alpha}$ for each $\alpha < \kappa$, Lemma \ref{l:CyclicContinuous} says that $\pair{e',1-e'} \leq \pair{f_\alpha,g_\alpha}$ for each $\alpha < \kappa$. This means that $\pair{e',1-e'}$ is a minimal right coprime pair below the initial chain, which concludes the proof.
\end{proof}

\begin{corollary}
Suppose that $R$ is a left continuous ring. Then $R$ is a right strongly exchange ring.
\end{corollary}

\begin{remark}\label{R:StronglyExchangeNonSelfInjective}
The converse of this result is not true. For instance, a commutative local domain is strongly exchange by Proposition \ref{p:CharacterizationRingsCoprimePairs}. But it is not continuous, unless it is a field, since it does not satisfy (C2). 

On the other hand, any von Neumann regular left continuous ring which is not left self-injective (see e.g. \cite[Example 13.8]{Goodearl91}) is another example of a right strongly exchange ring which is not left cotorsion.
\end{remark}


\section{Strongly exchange rings. Main properties}

We have proved in the previous section that local, left cotorsion and left continuous rings are right strongly exchange. All these rings are semiregular (see \cite[Theorem 6]{GuilHerzog} and \cite[Proposition 3.5]{MohamedMuller}). We are going to show in this section that right strongly exchange rings are, in general semiregular. Let us begin with a couple of technical lemmas that we will need later on.

\begin{lemma}\label{l:Intersection}
Let $\pair{a,b}$ be a right coprime pair such that $aR \cap bR \leq J(R)$. Then $(aR+J(R))\cap (bR+J(R)) \leq J(R)$.
\end{lemma}

\begin{proof}
Let $y \in (aR+J(R)) \cap (bR+J(R))$ and take $r_1,s_2 \in R$, $j_1,j_2 \in J(R)$ such that $y=ar_1+j_1=bs_2+j_2$.

Since $\pair{a,b}$ is a right coprime pair, there exist $s,r \in R$ such that $1=ar+bs$. Then $bsa=a(1-ra) \in aR \cap bR \leq J(R)$ and $arb=b(1-sb) \in aR \cap bR \leq J(R)$. Consequently, $(1-ar)ar_1 \in J(R)$.

On the other hand, multiplying the identity $ar_1 = bs_2+j_2-j_1$ by $ar$ and using that $arb \in J(R)$, we conclude that $arar_1 \in J(R)$.

Finally, $ar_1 = (1-ar)ar_1+arar_1 \in J(R)$, so that $y \in J(R)$ as well.
\end{proof}

\begin{lemma}\label{l:StrictCoprimePairsRadical}
Let $\pair{a,b}$ be a right coprime pair. 
\begin{enumerate}
\item If $aR \cap bR$ is not contained in $J(bR)$, then there exists a $z \in R$ such that $\pair{a,z}$ is a right coprime pair strictly less than $\pair{a,b}$.

\item If $aR \cap bR$ is not contained in $J(R)$, there exists a right coprime pair $\pair{x,y}$ strictly less than $\pair{a,b}$ and such that $x=a$ or $y=b$.
\end{enumerate}
\end{lemma}

\begin{proof}
(1) Choose an $x \in (aR \cap bR)\setminus J(bR)$. Then $xR$ is not superfluous in $bR$ and there exists a proper submodule $X \leq bR$ with $xR + X = bR$. In particular, there exists a $z \in X$ with $xR+zR=bR$. Then, as $x \in aR$, $aR+zR=R$, which means that $\pair{a,z}$ is a coprime pair satisfying $\pair{a,z} < \pair{a,b}$.

(2) Take $x \in (aR \cap bR)\setminus J(R)$. We claim that $x \notin J(aR) \cap J(bR)$. Suppose, on the contrary, that $x \in J(aR) \cap J(bR)$ and let $X$ be a right ideal of $R$ such that $xR+X = R$. Then, by modularity, $aR = xR+(aR \cap X)$ and $bR = xR+(bR\cap X)$, which implies that $aR = aR \cap X$ and $bR = bR \cap X$, since $xR$ is superfluous in $aR$ and $bR$. In particular, $R=aR+bR \leq X$. This means that $xR$ is superfluous in $R$ and thus, $x \in J(R)$. A contradiction which proves our claim.

Therefore, $x$ does not either belong to $J(aR)$ or $J(bR)$. Now, the result follows from (1).
\end{proof}

\begin{theorem}\label{t:ChainsImpliesSemiregular}
Every right strongly exchange ring is semiregular.
\end{theorem}

\begin{proof}
First, we claim that for any right coprime pair $\pair{a,b}$ there exists an element $x$ such that $\pair{a,x}$ is a right coprime pair, $\pair{a,x} \leq \pair{a,b}$ and $aR \cap xR \leq J(xR)$.

Suppose that our claim is false. We are going to construct by transfinite induction an element $b_\alpha$, for each ordinal $\alpha$, such that 
\begin{enumerate}[label=(\alph*)]
\item $\pair{a,b_\alpha}$ is a right coprime pair;

\item $a R \cap b_\alpha R$ is not contained in $J(b_\alpha R)$;

\item $\pair{a,b_{\alpha}} < \pair{a,b_{\gamma}}$, for each $\gamma < \alpha$;

\item if $\alpha>0$ is limit, then $b_\alpha$ is idempotent.
\end{enumerate}

For $\alpha = 0$, set $b_0=b$.

Let now $\alpha>0$ be an ordinal and assume that we have just constructed, for each $\gamma < \alpha$, elements  $b_\gamma\in R$ satisfying the above conditions. If $\alpha$ is successor, say $\alpha = \beta+1$, apply Lemma \ref{l:StrictCoprimePairsRadical} to get an element $b_{\alpha}$ such that $\pair{a,b_{\alpha}}$ is a right coprime pair strictly smaller than $\pair{a,b_\beta}$. Since we are assuming that our claim is false, $aR \cap b_{\alpha}R\nsubseteq J(b_{\alpha}R)$.

Suppose now that $\alpha$ is limit. If the set $\{\gamma < \alpha\mid \gamma \textrm{ is limit}\}$ is not cofinal in $\alpha$, then there exists a limit ordinal $\beta < \alpha$ such that $\alpha = \beta+\omega$. Then the descending chain of right coprime pairs, $\{\pair{a,b_{\beta+n}}\mid n < \omega\}$, is compatible and, by hypothesis, it has a minimal lower bound $\pair{x,b_\alpha}$. By Proposition \ref{p:MinimalCoprimePairs}, $b_\alpha$ can be chosen to be idempotent. Since $xR \leq  aR$, $\pair{a,b_\alpha}$ is a right coprime pair below the chain. Moreover, $\pair{a,b_\alpha} < \pair{a,b_\gamma}$ for each $\gamma < \alpha$, since the identity $\pair{a,b_\alpha}=\pair{a,b_\gamma}$ for some $\gamma < \alpha$, implies that $\pair{a,b_\gamma}=\pair{a,b_{\gamma+1}}$, a contradiction. Finally, $aR \cap b_\alpha R$ is not contained in $J(b_\alpha R)$ since we are supposing that our initial claim is false.

It remains to prove the case in which $\{\gamma < \alpha \mid \gamma \textrm{ is limit}\}$ is cofinal in $\alpha$. If this holds, we have the descending chain of right coprime pairs $\{\pair{a,b_\gamma}\mid \gamma < \alpha \textrm{ limit}\}$.  The arguments used in the proof of Proposition \ref{p:PropertiesChainsCoprimePairs} show that this chain is compatible. So we can find by hypothesis a minimal lower bound $\pair{x,b_\alpha}$, in which $b_\alpha$ idempotent by Proposition \ref{p:PropertiesChainsCoprimePairs}. Reasoning as above, we get that $\pair{a,b_\alpha}$ is a right coprime pair below the chain such that $aR \cap b_\alpha R\nsubseteq J(b_\alpha R)$. This finishes the construction.

We have constructed, for any ordinal $\alpha$, a right coprime pair $\pair{a,b_\alpha}$ such that $\pair{a,b_\alpha}<\pair{a,b_\beta}$ if $\alpha<\beta$. In particular, this means that $b_\alpha\neq b_\beta$ if $\alpha<\beta$ and therefore, the cardinality of $R$ must be at least the cardinality of $\alpha$ for any ordinal $\alpha$. A contradiction that proves our initial claim.

Finally, given any $a \in R$, we can apply our claim to the right coprime pair $\pair{a,1-a}$ to get an element $x\in R$ such that $\pair{a,x}$ is a right coprime pair with $aR \cap xR \leq J(xR)$ and $\pair{a,x}\leq \pair{a,1-a}$. By Lemma \ref{l:Intersection}, 
\begin{displaymath}
(a+J(R))R/J(R) \bigoplus (x+J(R))R/J(R)=R/J(R)
\end{displaymath}
which means that $(a+J(R))R/J(R)$ is a direct summand of $R/J(R)$. This proves that $R/J(R)$ is a von Neumann regular ring.

In order to prove that idempotents lift modulo the Jacobson radical of $R$, note that $R$ is an exchange ring and that exchange rings have this property \cite[Theorem 2.1 and Proposition 1.5]{Nicholson}.
\end{proof}

The converse of this result is not true, since there exist von Neumann regular rings which are not right strongly exchange (see Example \ref{e:RegularRingNotDescendingChain}).

\begin{remark}
	We have proved in Section 3 that local, left self-injective, left pure-injective, left cotorsion and left continuous rings are right strongly exchange. All these rings satisfy that they are left continuous modulo their Jacobson radical (see \cite[Theorem 8]{GuilHerzog} and \cite[Theorem 3.11]{MohamedMuller}). But we do not know whether any right strongly exchange rings enjoys this property. 
	
	If we strengthen the definition of right strongly exchange ring a little bit to include all compatible descending systems of right coprime pairs instead of just compatible descending chains of right coprime pairs, we can prove that the ring is left continuous modulo its Jacobson radical. However, we do not know whether any left continuous ring satisfies this last condition.
\end{remark}

Let us state the ideas of the preceding remark more precisely.
\begin{definition} Let $R$ be a ring.
	
\begin{enumerate}
\item A descending system of right coprime pairs is a downwards directed subset of $\RCP(R)$. I.e., a subset $\{p_i\mid i \in I\}\subseteq \RCP(R)$ indexed by a directed set $I$ such that $p_j \leq p_i$ whenever $i \leq j$.

\item A descending system of right coprime pairs $\{p_i\mid i \in I\}$ is called compatible if for each $i \in I$, there exists a pair of generators $(a_i,b_i)$ of $p_i$, and families of scalars $\{r_{ij}\mid i < j\}$ and $\{s_{ij} \mid i < j\}$, satisfying that
\begin{enumerate}
\item $a_j=a_ir_{ij}$ and $b_j=b_is_{ij}$;

\item $r_{ik}=r_{ij}r_{jk}$ and $s_{ik}=s_{ij}s_{jk}$.
\end{enumerate}
\end{enumerate}

for each $i < j < k$ in $I$.
\end{definition}
\medskip

As in the case of chains, if $\{\pair{a_i,b_i}\mid i \in I\}$ is a compatible descending system of right coprime pairs, we will assume that $(a_i,b_i)$ is the pair of generators satisfying the compatibility condition of the preceding definition.

Recall that the \textit{singular submodule} of a right $R$-module $M$ is its submodule
\begin{displaymath}
\Z(M)=\{x \in M \mid {\mathbf r}_R(x) \textrm{ is essential in } R_R\}.
\end{displaymath}
A module $M$ is called \textit{singular} if $M=\Z(M)$ and \textit{non-singular} if $\Z(M)=0$.

\begin{lemma}\label{l:InclusionsInEssentialExtensions}
Let $M$ be a non-singular module and $K$, $N$ and $L$, submodules of $M$ such that both inclusions $K \leq N$ and $K \leq L$ are essential. If $\frac{M}{L}$ is non-singular (for instance, if $L$ is a direct summand of $M$), then $N \leq L$.
\end{lemma}

\begin{proof}
Since $K \leq N \cap L \leq N$, we get that $N \cap L$ is essential in $N$. And, by \cite[Proposition 1.21]{Goodearl76}, $\frac{N}{N \cap L}$ is singular. On the other hand, $\frac{N}{N \cap L} \cong \frac{N+L}{L}$ is non-singular, since it is a submodule of the non-singular module $\frac{M}{L}$. Therefore, $\frac{N}{N \cap L}=0$ and, consequently, $N \leq L$.
\end{proof}

\begin{lemma}\label{l:EssentialInfR}
Let $\{e_i \mid i \in I\}$ be a family of idempotents of $R$ such that $\sum_{i \in I}e_iR$ is essential in $R_R$. Let $J \subseteq I$ and $f\in R$, an idempotent such that $\sum_{j \in J}e_jR \leq fR$. If $fe_i=0$ for each $i \in I\setminus J$, then $\sum_{j \in J}e_jR$ is essential in $fR$.
\end{lemma}

\begin{proof}
Let us choose an $r \in R$ such that $fr \neq 0$. Since $\sum_{i \in I}e_iR$ is essential in $R$, there exists an $s \in R$ such that $0 \neq frs \in \sum_{i \in I}e_iR$. Write 
\begin{equation}\label{e:SumIdempotents}
frs = \sum_{i \in I}e_ir_i
\end{equation}
for a family $\{r_i\mid i \in I\}\subseteq R$ satisfying that $\{i \in I\mid r_i \neq 0\}$ is finite. Multiplying this identity by $f$ on the left we obtain that
\begin{displaymath}
frs = \sum_{j \in J}fe_jr_j.
\end{displaymath}
And, as $fe_j = e_j$ for each $j \in J$, we deduce that $fsr \in \sum_{j \in J}e_jR$.
\end{proof}

\begin{lemma}\label{l:EssentialInN}
Let $R$ be a von Neumann regular ring, $N$ a non-zero right ideal of $R$ and $X$ a set of idempotent elements in $N$ such that $\sum_{e \in X}eR$ is direct. Then $X$ is contained in a set of idempotents $X'$ of $R$ such that $\sum_{e \in X'}e R$ is direct and essential in $N$.
\end{lemma}

\begin{proof}
Consider the set
\begin{displaymath}
\mathcal L = \left\{F \subseteq N \mid \textrm{ $F$ consists of idempotents, $\sum_{e \in F}eR$ is direct and $X \subseteq F$}\right\}
\end{displaymath}
$\mathcal L$ is an inductive non-empty partially ordered set. Thus, it has a maximal element $X'$ by Zorn's lemma. Let us check that $\sum_{e \in X'}eR$ is essential in $N$. For any non-zero $x \in N$, there exists an idempotent $f$ such that $xR = fR$. By construction, the sum $\sum_{e \in X'}eR + fR$ is not direct, so there exists an $r \in R$ such that $0 \neq fr \in \sum_{e \in X'}eR$. In particular, $xR \cap \sum_{e \in X'}eR \neq 0$.
\end{proof}

\begin{theorem}\label{t:StronglyImpliesContinuousModuloTheRadical}
Let $R$ be a ring. If every compatible descending system of right coprime pairs of $R$ has a minimal lower bound, then $R/J(R)$ is left continuous.
\end{theorem}

\begin{proof}
Denote $R/J(R)$ by $\overline R$. For every $x \in R$, let $\overline x$ be the corresponding element in $R/J(R)$. And, for each subset $X$ of $R$, denote by $\overline{X}$ the subset $\{\overline x\mid x \in X\}\subseteq \overline R$.

Let $\overline{I}$ be a left ideal of $\overline R$. Using Lemma \ref{l:EssentialInN}, and the fact that $R$ is semiregular by Theorem \ref{t:ChainsImpliesSemiregular}, we can find sets of idempotents of $R$, $E$ and $F$, such that $\{\overline R\overline e  \mid \overline e \in \overline E\}\cup \{\overline R \overline f  \mid \overline f \in \overline F\}$ is independent, $\sum_{\overline e \in \overline E}\overline R \overline e$ is essential in $\overline I$ and $\sum_{\overline e \in \overline E}\overline R \overline e +\sum_{\overline f \in \overline F}\overline R \overline f$ is essential in $\overline R$. We may apply now \cite[Lemma 13]{ZimmermannZimmermann} to conclude that $\{Re\mid e \in E\} \cup \{Rf \mid f \in E\}$ is a local direct summand of ${_R}R$. I.e., it is independent and for every finite subset $X \subseteq E \cup F$, $\sum_{e \in X}Re$ is a direct summand of $R$.

For each finite subset $X$ of $E$, fix an idempotent $e_X$ such that $\sum_{e \in X}Re=Re_X$. Similarly, fix and idempotent $f_Y$ for each finite subset $Y$ of $F$. Then, for every $X \subseteq E$ and $Y \subseteq F$ finite, $Re_X \cap Rf_Y=0$ and $Re_X + Rf_Y$ is a direct summand, so that $\pair{1-e_X,1-f_Y}$ is a right coprime pair by \cite[Lemma 12]{ZimmermannZimmermann}. It is easy to check that the family
\begin{displaymath}
\{\pair{1-e_X,1-f_Y}\mid X \subseteq E, Y \subseteq E \textrm{ are finite subsets}\}
\end{displaymath}
is a compatible descending system of right coprime pairs.

Now, we may apply our hypothesis to find a minimal lower bound $\pair{a,b}$ of the system which, by Proposition \ref{p:MinimalCoprimePairs}, must be of the form $\pair{g,1-g}$, for some idempotent element $g$. Then, $g \in (1-e_X)R$ and $1-g \in (1-f_Y)R$ for each $X \subseteq E$ and $Y \subseteq F$ finite, which implies that $e_X \in R(1-g)$ and $f_Y \in Rg$. In particular, $\overline e \in \overline R(1-\overline g)$ and $\overline f \in \overline R \overline g$ for each $e \in E$ and $f \in F$. Since this implies that $\overline f (1-\overline g)=0$ for each $f \in F$, Lemma \ref{l:EssentialInfR} claims that $\sum_{e \in E}\overline R \overline e$ is essential in $\overline R (1-\overline g)$. Finally, $\overline I$ is essential in $\overline R (1-\overline g)$ by Lemma \ref{l:InclusionsInEssentialExtensions}.

We have just proved that $\overline R$ satisfies (C1). Since $\overline R$ is von Neumann regular by Theorem \ref{t:ChainsImpliesSemiregular}, it trivially satisfies (C2). Therefore, $\overline R$ is left continuous.
\end{proof}

We use some ideas in the proof of the preceding result to show that the strong exchange property is not left-right symmetric.

\begin{example}\label{e:NonSymmetric}
Let $D$ be a division ring and $V$ a vector space over $D$ with infinite countable dimension. Denote by $S$ the endomorphism ring of $V$. Then $S$ is left strongly exchange but not right strongly exchange.
\end{example}

\begin{proof}
It is well known, e. g. \cite[Proposition 2.23]{Goodearl76}, that $S$ is right self-injective. By Theorem \ref{t:CotorsionAreStrongly}, $S$ is left strongly exchange.

Let $\{v_n \mid n < \omega\}$ be a basis of $V$ and denote by $e_n:V \rightarrow V$ the endomorphism of $V$ satisfying $e_n(v_i)=0$ if $i \neq n$ and $e_n(v_n)=v_n$, for each $n < \omega$. Moreover, let $p:V \rightarrow V$ be the endomorphism of $V$ such that $p(e_i)=e_0$ for each $i < \omega$. As it is shown in the proof of \cite[Propostion 2.23]{Goodearl76}, $\left(\bigoplus_{n < \omega}Se_n \right) \cap Sp=0$, which in particular implies, as a consequence of Proposition \ref{p:CharacterizationCoprimePairs}, that $\pair{1-\sum_{n \leq m}e_n,1-p}$ is a right coprime pair in $S$ for each $m < \omega$. Then, it is easy to see that $\left\{\pair{1-\sum_{n \leq m}e_n,1-p}\mid m < \omega\right\}$ is a (compatible) descending chain of right coprime pairs in $S$.

We are going to prove that this chain does not have a minimal lower bound. Assume, on the contrary, that there exists an idempotent $e \in S$ such that $\pair{e,1-e} \leq \pair{1-\sum_{n \leq m}e_n,1-p}$ for each $m < \omega$. Then, $\left(\sum_{n \leq m}e_n\right)e=0$ for each $m < \omega$, which implies that $\sum_{n \leq m}e_n \in S(1-e)$. 

Now we claim that $\{(1-e)(v_n)\mid n < \omega\}$ is linearly independent. Assume, in order to get a contradiction, that there exists $m < \omega$ such that $(1-e)(v_m)=\sum_{i=1}^n(1-e)(v_{m_i})d_i$ for some $m_1 \ldots, m_n \in \omega-\{m\}$ and $d_1, \ldots, d_n \in D$. Then, since $e_m = s(1-e)$ for some $s \in S$, we get
\begin{displaymath}
v_m = e_m(v_m) = s(1-e)(v_m) = \sum_{i=1}^ns(1-e)(v_{m_i})d_i = \sum_{i=1}^ne_m(v_{m_i})d_i=0,
\end{displaymath}
which is a contradiction. This proves our claim.

Finally, this claim implies that $1-e$ is monic and, since it is idempotent, $e=0$. But $1=1-e \in S(1-p)$, which implies that $p=0$ as well. This is a contradiction.
\end{proof}

Now we give a sufficient condition for a strongly exchange ring to be semiperfect.

\begin{theorem}\label{t:semiperfect}
Suppose that $R$ is a right strongly exchange ring with countably many idempotents. Then $R$ is semiperfect.
\end{theorem}

\begin{proof}
Let us first note that $R$ is semiregular by Theorem~\ref{t:ChainsImpliesSemiregular} and that $R/J(R)$ has countably many idempotents. Suppose that $R$ is not semiperfect. Then $R/J(R)$ is not semisimple and thus, it has infinite right Goldie dimension. This means that there exists a countable infinite family $\{e_n\mid n < \omega\}$ of non-zero idempotents in $R$ such that the family $\left\{(e_n+J(R))R/J(R) \mid n < \omega\right\}$ is independent. By \cite[Proposition 2.14]{Goodearl91}, we may assume that the family $\{e_n+J(R)\mid n < \omega\}$ is orthogonal.

For each infinite subset $A\subset\omega$ such that $\omega\setminus A$ is infinite, write $A=\{i_n(A)\mid n < \omega\}$ and $\omega\setminus A=\{j_n(A) \mid n < \omega\}$. Define the idempotents
\begin{displaymath}
x_n^A=1-\sum_{k=0}^n(e_{i_k(A)} +J(R))\, \textrm{ and } \, y_n^A=1-\sum_{k=0}^n(e_{j_k(A)}+J(R))
\end{displaymath}
for any $n<\omega$. We get a descending chain of right coprime pairs, $\{\pair{x_n^A,y_n^A} \mid n < \omega\}$, for which we can find, by hypothesis, a lower bound $\pair{x^A,y^A}$. Since $R/J(R)$ is von Neumann regular, we can assume that $x^A$ and $y^A$ are idempotent elements.

We claim that if $A$ and $B$ are distinct subsets of $\omega$ such that $\omega\setminus A$ and $\omega\setminus B$ are infinite, then $x^A \neq x^B$. This will prove the result, since this would imply that the set of idempotents of $R$ is uncountable.

To prove our claim, choose an $n \in A\setminus B$. Then $n \in \omega\setminus B$ and thus, there exist $u,v < \omega$ such that $n = i_u(A)$ and $n = j_v(B)$. Using that $x^A \in x_u^AR/J(R)$ and $(e_n+J(R))x^A_u = 0$, we get that $(e_n+J(R))x^A=0$. Similarly, $(e_n+J(R))y^B=0$, since $y^B \in y_v^BR/J(R)$ and $(e_n+J(R))y_v^B=0$. Now, writing $1+J(R)=x^Br+y^Bs$ for $r,s \in R$, we get that $e_n+J(R)=(e_n+J(R))x^Br$, from which we deduce that $(e_n+J(R))x^B\neq 0$. As $(e_n+J(R))x^A=0$, we conclude that $x^A\neq x^B$.
\end{proof}

\begin{corollary}\label{countable}
Any right strongly exchange ring with countably many idempotent elements and zero Jacobson radical is semisimple.
In particular, any regular left continuous ring with countably many idempotents is semisimple.
\end{corollary}

\begin{proof}
$R$ is semiperfect by the above theorem. So it is semisimple as its Jacobson radical is zero. Finally, note that, if $R$ is a left continuous von Neumann regular ring, then it satisfies these conditions by Theorem~\ref{t:ContinuousHaveCompatibleLowerBounds}.
\end{proof}

\begin{remark}
The following alternative proof of the last assertion of this corollary was communicated to us by K. Goodearl. Suppose that $R$ is a left continuous ring which is not semisimple. Then, by \cite[Corollary 2.16]{Goodearl91}, $R$ has an infinite set of non-zero orthogonal idempotents, $\{e_n\mid n < \omega\}$. Let $Q$ be the maximal left quotient ring of $R$. Since $Q$ is left self-injective, for every subset $J$ of $\omega$, there exists an idempotent $e_J$ of $Q$ such that $Qe_J$ is an injective hull of $\bigoplus_{j \in J}Qe_j$. The set $\{e_J \mid J \subseteq \omega\}$ is uncountable, which is a contradiction, since $R$ contains all the idempotent elements of $Q$ by \cite[Theorem 13.13]{Goodearl91}.
\end{remark}

We close this paper by suggesting some new lines of research. As we mentioned in the introduction, it is not known whether a module $M$ satisfying the finite exchange property does satisfy the full exchange property. On the other hand,  Warfield proved that a module $M$ satisfies the finite exchange property if and only if its endomorphism ring is exchange. Therefore, it is natural to ask whether a module $M$ satisfies the full exchange property provided its endomorphism ring is right strongly exchange. Let us call a module $M$ \textit{strongly exchange} when its endomorphism ring is right strongly exchange. The above question can rephrased as follows:

\begin{question}
Does every strongly exchange module satisfy the full exchange property?
\end{question}

A ring $R$ is called \textit{clean} when every element in $R$ is the sum of an idempotent and a unit. Suppose that $R$ satisfies that compatible descending systems of right coprime pairs have minimal lower bounds. Then, by Theorems \ref{t:ChainsImpliesSemiregular} and \ref{t:StronglyImpliesContinuousModuloTheRadical}, $R$ is semiregular and $R/J(R)$ is left continuous. By \cite[Theorem 3.9]{CamilloKhuranaLamNicholsonZhou}, $R/J(R)$ is clean and by \cite[Proposition 7]{CamilloYou} $R$ is clean as well. Moreover, note that all the examples of right strongly exchange rings exhibited in this paper are clean (for instance, left self injective, left cotorsion, local or left continuous rings). We may pose then the following question:

\begin{question}
Is every right strongly exchange ring clean?
\end{question}

\bigskip

{\bf Acknowledgement.} The authors would like to thank Prof. Pere Ara for several helpful comments and remarks. They also want to thank the referee for several comments and suggestions that improved the quality of this paper. 
	
\bibliographystyle{plain} \bibliography{references}

\end{document}